\documentclass[12pt]{article}
\usepackage{amsmath,amsthm,amssymb}
\usepackage{mathrsfs}
\usepackage[titletoc]{appendix}
\usepackage{bbm}
\usepackage{float}
\usepackage{tikz}
\begin{document}
\input amssym.def
\newcommand{\singlespace}{
    \renewcommand{\baselinestretch}{1}
\large\normalsize}
\newcommand{\doublespace}{
   \renewcommand{\baselinestretch}{1.2}
   \large\normalsize}
\renewcommand{\theequation}{\thesection.\arabic{equation}}

\def \ten#1{_{{}_{\scriptstyle#1}}}
\def \Z{\mathbb Z}
\def \C{\mathbb C}
\def \R{\mathbb R}
\def \Q{\mathbb Q}
\def \N{\mathbb N}
\def \F{\mathbb F}
\def \l{\lambda}
\def \V{V^{\natural}}
\def \wt{{\rm wt}}
\def \tr{{\rm tr}}
\def \Res{{\rm Res}}
\def \End{{\rm End}}
\def \Aut{{\rm Aut}}
\def \mod{{\rm mod}}
\def \Hom{{\rm Hom}}
\def \im{{\rm im}}
\def \<{\langle}
\def \>{\rangle}
\def \w{\omega}
\def \c{{\tilde{c}}}
\def \o{\omega}
\def \t{\tau }
\def \ch{{\rm ch}}
\def \a{\alpha }
\def \b{\beta}
\def \e{\epsilon }
\def \la{\lambda }
\def \om{\omega }
\def \O{\Omega}
\def \qed{\mbox{ $\square$}}
\def \pf{\noindent {\bf Proof: \,}}
\def \voa{vertex operator algebra\ }
\def \voas{vertex operator algebras\ }
\def \p{\partial}
\def \1{{\bf 1}}
\def \ll{{\tilde{\lambda}}}
\def \H{{\bf H}}
\def \F{{\Bbb F}}
\def \h{{\frak h}}
\def \g{{\frak g}}
\def \rank{{\rm rank}}
\def \({{\rm (}}
\def \){{\rm )}}
\def \Y {\mathcal{Y}}
\def \I {\mathcal{I}}
\def \A {\mathcal{A}}
\def \B {\mathcal {B}}
\def \Cc {\mathcal {C}}
\def \H {\mathcal{H}}
\def \M {\mathcal{M}}
\def \V {\mathcal{V}}
\def \L{\mathcal L}
\def \O{{\bf O}}
\def \1{{\bf 1}}
\def\Ve{V^{0}}
\def\ha{\frac{1}{2}}
\def\se{\frac{1}{16}}
\def\g{\mathfrak g}
\def\h{\mathfrak h}
\def\wh{\widehat{\mathfrak h}}
\def\wv{\widehat{V}}
\def\ww{\widehat{W}}
\def\wg{\widehat{\mathfrak g}}
\def\mraff{\mathrm{aff}}
\def\J{\mathcal J}
\def\E{{\mathbb E}}
\def\s{{\bf s}}
\def\S{{\mathbb S}}
\singlespace
\newtheorem{thm}{Theorem}[section]
\newtheorem{prop}[thm]{Proposition}
\newtheorem{lem}[thm]{Lemma}
\newtheorem{cor}[thm]{Corollary}
\newtheorem{rem}[thm]{Remark}
\newtheorem{con}[thm]{Conjecture}
\newtheorem*{CPM}{Theorem}
\newtheorem{definition}[thm]{Definition}
\newtheorem{de}[thm]{Definition}

\begin{center}
{\Large {\bf  Representations of the parafermion vertex operator algebras }} \\

\vspace{0.5cm} Chongying Dong\footnote{Supported by NSF grant DMS-1404741 and China NSF grant 11371261}
\\
School of Mathematics, Sichuan University,
 Chengdu 610064 China \& \\
 Department of Mathematics, University of
California, Santa Cruz, CA 95064 USA \\
Li Ren\footnote{Supported by China NSF grant 11301356}\\
 School of Mathematics,  Sichuan University,
Chengdu 610064 China

\end{center}
\hspace{1.5 cm}

\begin{abstract}
The rationality of the parafermion vertex operator algebra $K(\g,k)$ associated to any finite dimensional simple Lie algebra $\g$ and any nonnegative integer $k$ is established and
the irreducible modules are determined.
\end{abstract}

\section{Introduction}

 This paper deals with the representation theory of the parafermion vertex operator algebra $K(\g,k)$ associated to any finite dimensional simple Lie algebra $\g$ and positive integer $k.$ Using the recent results on abelian orbifolds  \cite{CM}, \cite{M} and coset construction \cite{KM}, we establish the rationality of $K(\g,k)$ and determine the irreducible modules. If
$\g=sl_2,$ the classification of irreducible modules and rationality were achieved previously in
\cite{ALY1}- \cite{ALY2}.

The origin of the  parafermion vertex operator algebra is the  theory of $Z$-algebra developed in \cite{LP,LW1,LW2} for constructing irreducible highest weight modules for affine  Kac-Moody algebra. The main idea was to determine the vacuum space of a module for an affine Kac-Moody algebra. The vacuum space is the space of highest weight vectors for the Heisenberg algebra  and is a module for the $Z$-algebra. Using the language of the conformal field theory and vertex operator algebra, this is the coset theory associated to the simple affine vertex operator algebra $L_{\wg}(k,0)$ and its vertex operator  subalgebra $M_{\wh}(k)$ generated by the Cartan subalgebra $\h$ of $\g.$ That is, the parafermion vertex operator algebra $K(\g,k)$ is the commutant of $M_{\wh}(k)$ in $L_{\wg}(k,0)$ or a coset construction associated to the Lie algebra $\g$ and
its Cartan subalgebra  $\h$ \cite{GKO}.

There have been a lot of investigations of the parafermion vertex operator algebras and parafermion conformal field theory in physics (see \cite{BEHHH}, \cite{G}, \cite{GQ}, \cite{H}, \cite{WD}, \cite{We}, \cite{ZF}).
The relation  between the parafermion conformal field theory and the $Z$-algebra has been clarified using the generalized vertex operator algebra in \cite{DL}. But a systematic study of the parafermion vertex operator algebra took place only very recently. A set of generators of the  parafermion vertex operator algebra $K(\g,k)$ was determined in \cite{DLY}, \cite{DLWY} and \cite{DW1}.  Moreover, the parafermion vertex operator algebra $K(\g,k)$ for an arbitrary $\g$ is generated by $K(\g^{\alpha},k_{\alpha})$ where $\g^{\alpha}$ is the 3-dimensional subalgebra isomorphic to $sl_2$ and associated to a root $\alpha,$ and $k_{\alpha}=\frac{2}{\<\alpha,\alpha\>}k.$  From this point of view, $K(sl_2,k)$ is the building block of the general parafermion vertex operator algebra $K(\g,k).$

The representation theory of $K(\g,k)$ has not been understood well. It is proved in \cite{ALY1} that
the Zhu algebra $A(K(sl_2,k))$ is a finite dimensional semisimple associative algebra, the irreducible
modules of $K(sl_2,k)$ are exactly those which appear in the integrable highest weight modules
of level $k$ for affine Kac-Moody algebra $\widehat{sl_2},$ and $K(\g,k)$ is $C_2$-cofinite
for any $\g.$ Also see \cite{DW2} on the $C_2$-cofiniteness of $K(\g,k).$

It is well known that the $L_{\wg}(k,0)$ is a rational vertex operator algebra \cite{FZ}, \cite{L2} and $K(\g,k)$ is also the commutant of lattice vertex operator algebra $V_{\sqrt{k}Q_L}$ in $L_{\wg}(k,0)$ where $Q_L$ is the lattice spanned by the long roots \cite{DLY}, \cite{DW3}. There is a famous conjecture in the
coset construction which says that the commutant $U^c$ of a rational vertex operator subalgebra $U$  in a rational vertex operator algebra $V$ is also rational. So it is widely believed that $K(\g,k)$ should give a new class of rational vertex operator algebras although this can only been proved so far in the case $\g=sl_2$ and
$k\leq 6$ \cite{DLY}.

Here are the main results in this paper: (1) For any simple Lie algebra $\g$ and any positive integer $k$,
the parafermion vertex operator algebra $K(\g,k)$ is rational, (2) The irreducible modules of $K(\g,k)$
are those appeared in the integrable highest weight modules
of level $k$ for affine Kac-Moody algebra $\wg$ and we give some identification among these irreducible  $K(\g,k)$-modules.

Although the parafermion vertex operator algebras $K(\g,k)$ are subalgebras of the affine vertex operator algebras $L_{\wg}(k,0)$ their structure are much more complicated. According to \cite{DLY}, $K(sl_2,k)$ is strongly generated by the Virasoro element $\omega,$ and highest weight vectors $W^3,W^4,W^5$ of weights $3,4,5.$ For each root $\alpha$ let $\omega_{\alpha}, W^i_{\alpha}$ be the corresponding elements in $K(\g^{\alpha},k_{\alpha}).$ Then $K(\g,k)$ is generated by $\omega_{\alpha},$ $W^i_{\alpha}$ for all positive roots $\alpha$ and $i=3,4,5.$ Although the Lie bracket $[Y(u,z_1),Y(v,z_2)]$ is complicated for $u,v\in S_{\alpha}=\{\omega_\alpha, W^i_{\alpha}|i=3,4,5\}$ but computable \cite{DLY}. But if $\alpha, \beta$ are two different positive roots and $u\in S_{\alpha}, v\in S_{\beta}$ we do not know the commutator relation $[Y(u,z_1),Y(v,z_2)]$ in general.  So it is very difficult to prove the main results directly by using the generators and relations.

The rationality of $K(\g,k)$ is relatively easy due to recent results in \cite{CM}, \cite{M}. Let $Q_L$ be the sublattice of root lattice of $\g$ spanned by the long roots. Then $V_{\sqrt{k}Q_L}\otimes K(\g,k)$ is a vertex operator subalgebra of $L_{\wg}(k,0).$ Moreover, $V_{\sqrt{k}Q_L}\otimes K(\g,k)$  can be realized as fixed point subalgebra $L_{\wg}(k,0)^G$ for a finite abelian group $G$. It follows from \cite{CM} that
$V_{\sqrt{k}Q_L}\otimes K(\g,k)$ is rational. Using the rationality of $V_{\sqrt{k}Q_L},$ one can easily conclude that  $K(\g,k)$ is rational.

It follows from \cite{KM} each irreducible $K(\g,k)$-module occurs in an irreducible $\wg$-module $L_{\wg}(k,\Lambda)$
where $\Lambda$ is a dominant weight of finite dimensional Lie algebra $\g$ satisfying $\<\Lambda,\theta\>\leq k,$ and $\theta$ is the maximal root of $\g.$ We denote these irreducible modules by $M^{\Lambda, \lambda}$
in this paper where $\lambda$ lies in $\Lambda+Q$ and $Q$ is the root lattice of $\g.$  We also find some identification of these irreducible $K(\g,k)$-modules by using the simple currents \cite{L2'} and \cite{L4}.
It turns out these identifications are complete \cite{ADJR}. The quantum dimensions and the fusion rules
have also been determined in \cite{DW4} and \cite{ADJR}.

The connection between the parafermion vertex operator algebras and the commutants of $L_{\wg}(k,0)$ in $L_{\wg}(1,0)^{\otimes k}$ has also been investigated in \cite{LY}, \cite{JL1}, \cite{JL2}.

The paper is organized as follows. We review the vertex operator algebra $V_{\wg}(k,0)$ associated to affine Kac-Moody algebra $\wg$ and its irreducible quotient $L_{\wg}(k,0)$ in Section 2. We also discuss the $V_{\wg}(k,0)$-modules $V_{\wg}(k,\Lambda)$ generated by any irreducible highest weight $\g$-module $L_\g(\Lambda).$ We recall from \cite{DW1}  the vertex operator algebra $N(\g,k)$ which is the commutant of the  Heisenberg vertex operator algebra $M_{\wh}(k)$ in $V_{\wg}(k,0).$ We also decompose each
$V_{\wg}(k,0)$-module $V_{\wg}(k,\Lambda)$ into a direct sum of $M_{\wh}(k)\otimes N(\g,k)$-modules. Section 4 is devoted to the study of the  parafermion vertex operator algebra $K(\g,k).$ In particular, $K(\g,k)$ is the irreducible quotient of $N(\g,k)$ and is also the commutant of $M_{\wh}(k)$ in $L_{\wg}(k,0).$ We also decompose each irreducible $L_{\wg}(k,0)$-module $L_{\wg}(k,\Lambda)$ into a direct sum of irreducible $M_{\wh}(k)\otimes K(\g,k)$-modules $M_{\wh}(k,\lambda)\otimes M^{\Lambda,\lambda}$ for $\lambda\in \Lambda+Q$ where $Q$ is the root lattice of $\g.$ In addition, we  investigate the relation between these irreducible $K(\g,k)$-modules $M^{\Lambda,\lambda}.$   Both rationality and classification of irreducible modules for $K(\g,k)$ for any simple Lie algebra $\g$ are obtained in Section 5.

We assume the readers are familiar with the admissible modules, rationality, $A(V)$-theory as presented in \cite{DLM1}, \cite{DLM2} and \cite{Z}.

We thank Professors Tomoyuki Arakawa, Cuipo Jiang and Haisheng Li for many valuable discussions and suggestions.

\section{Affine vertex operator algebras}
\setcounter{equation}{0}

In this section, we recall the basics of vertex operator algebras
associated to affine Lie algebras following \cite{FZ} and
\cite{LL}.

Fix a finite dimensional simple Lie algebra  $\g$ with a Cartan
subalgebra $\h.$ Denote the corresponding root system by $\Delta$
and the root lattice by $Q.$ Let $\< ,\>$ be an invariant symmetric
nondegenerate bilinear form on $\g$ such that $\<\a,\a\>=2$ if
$\alpha$ is a long root, where we have identified $\h$ with $\h^*$
via $\<,\>.$ We denote the image of $\alpha\in
\h^*$ in $\h$ by $t_\alpha.$ That is, $\alpha(h)=\<t_\alpha,h\>$
for any $h\in\h.$ Fix simple roots $\{\alpha_1,...,\alpha_l\}$
and let $\Delta_+$ be the set of corresponding positive roots.
Denote the highest root by $\theta.$

For $\alpha\in \Delta_+$ we denote the root space by $\g_{\alpha}$
and  fix $x_{\pm
\alpha}\in \g_{\pm \alpha}$ and
$h_{\alpha}=\frac{2}{\<\a,\a\>}t_\alpha\in \h$ such that
 $[x_\a,x_{-\a}]=h_{\a}, [h_\a,x_{\pm \a}]=\pm 2x_{\pm\a}.$ That
is, $\g^{\a}=\C x_{\a}+\C h_{\alpha}+\C x_{-\alpha}$ is isomorphic
to $sl_2$ by sending $x_\a$ to $\left(\begin{array}{ll} 0 & 1\\ 0 &
0\end{array}\right),$ $x_{-\a}$ to $\left(\begin{array}{ll} 0 & 0\\
1 & 0\end{array}\right)$ and $h_\a$ to $\left(\begin{array}{ll} 1 &
0\\ 0 & -1\end{array}\right).$ Then
$\<h_\a,h_\a\>=2\frac{\<\theta,\theta\>}{\<\alpha,\alpha\>}$ and
$\<x_{\a},x_{-\a}\>=\frac{\<\theta,\theta\>}{\<\alpha,\alpha\>}$
for all
$\alpha\in \Delta.$

Recall that $\widehat{\mathfrak g}= \g \otimes
\C[t,t^{-1}] \oplus \C K$ is the affine Lie algebra associated to $\g$ with Lie bracket
$$[a(m), b(n)] = [a,b](m+n) + m \< a,b \> \delta_{m+n,0}K, [K, \widehat{\mathfrak g}]=0$$
for $a, b \in \g$ and $m,n\in \Z$ where $a(m)=a \otimes t^m$.
Note that $\widehat{\h}=\h \otimes
\C[t,t^{-1}] \oplus \C K$ is a subalgebra of $\widehat{\g}$.

Fix a
positive integer $k$ and a weight $\Lambda\in \h^*.$ Let
$L_\g(\Lambda)$ be the irreducible highest weight module for $\g$
with highest weight $\Lambda$ and
\begin{equation*}
 V_{\widehat{\g}}(k,\Lambda) = Ind_{\g \otimes \C[t]\oplus \C
K}^{\widehat{\g}}L_\g(\Lambda)
\end{equation*}
be the induced $\widehat{\g}$-module where ${\g} \otimes \C[t]t$
acts as $0,$ $\g=\g\otimes t^0$ acts as $\g$ and $K$ acts as $k$
on $L_\g(\Lambda)$. Then $V_{\widehat{\g}}(k,\Lambda)$ has a unique maximal
submodule $\J(k,\Lambda)$ and we denote the irreducible quotient by
$L_{\widehat{\g}}(k,\Lambda)$. In the case $\Lambda=0,$ the maximal
submodule $\J=\J(k,0)$ of $V_{\widehat{\g}}(k,0)$ is generated by
$x_{\theta}(-1)^{k+1}1$ \cite{K} where $1=1\otimes 1\in V_{\widehat{\g}}(k,0).$ The $\J$ is also generated by
$x_{-\theta}(0)^{k+1}x_{\theta}(-1)^{k+1}1$ \cite{DLWY}, \cite{DW1}.
Moreover, $ L_{\widehat{\g}}(k,\Lambda)$ is integrable if
and only if $\Lambda$ is a dominant weight such that
$\<\Lambda,\theta\>\leq k$ \cite{K}. Denote the set of such $\Lambda$ by $P^k_+.$

It is well known that
$V_{\widehat{\g}}(k,0)$ is a vertex operator algebra generated by $a(-1)\1$ for
$a\in \g$ such that
$$Y(a(-1)\1,z) = a(z)=\sum_{n \in \Z} a(n)z^{-n-1}$$
 with the vacuum vector $\1=1$ and the
Virasoro vector
\begin{align*}
\omega_{\mraff} &= \frac{1}{2(k+h^{\vee})} \Big(
\sum_{i=1}^{l}u_i(-1)u_i(-1)\1 +\sum_{ \alpha\in\Delta}
\frac{\<\a,\a\>}{2}x_{\alpha}(-1)x_{-\alpha}(-1)\1 \Big)
\end{align*}
of central charge $\frac{k\dim \g}{k+h^{\vee}}$ (see \cite{FZ}, \cite{LL}), where $h^{\vee}$ is the dual
Coxeter number of $\g$ and $\{u_i|i=1,\ldots,l\}$ is an
orthonormal basis of $\mathfrak h.$ Moreover, each $V_{\widehat{\g}}(k,\Lambda)$
is a module for $V_{\widehat{\g}}(k,0)$ for any $\Lambda$ \cite{FZ}, \cite{LL}. As usual, we let $Y(\omega_{\mraff},z)=\sum_{n\in\Z}L_{\mraff}(n)z^{-n-2}.$
Then
$$V_{\widehat{\g}}(k,\Lambda)=\bigoplus_{n\geq 0}V_{\widehat{\g}}(k,\Lambda)_{n_{\Lambda}+n}$$
where $V_{\widehat{\g}}(k,\Lambda)_{n_{\Lambda}+n}=\{v\in V_{\widehat{\g}}(k,\Lambda)|L_{\mraff}(0)v=(n_{\Lambda}+n)v\},$ $n_{\Lambda}=\frac{(\Lambda,\Lambda+2\rho)}{2(k+h^{\vee})}$ and
$\rho=\frac{1}{2}\sum_{\alpha\in \Delta_+}\alpha.$ The top level $V_{\widehat{\g}}(k,\Lambda)_{n_{\Lambda}}$ of $V_{\widehat{\g}}(k,\Lambda)$ is exactly the $L_\g(\Lambda).$

Now $\J$ is the maximal ideal of $V_{\widehat{\g}}(k,0)$ and
 $L_{\widehat{\g}}(k,0)$ is a simple, rational vertex operator algebra
such that the irreducible $L_{\widehat{\g}}(k,0)$-modules are exactly
the integrable highest weight $\widehat{\g}$-modules $L_{\wg}(K,\Lambda)$ of level $k$ for
$\Lambda\in P_+^k$ (cf. \cite{DL}, \cite{FZ}, \cite{LL}).

\section{Vertex operator algebras $N(\g,k)$ }
\setcounter{equation}{0}

We recall from  \cite{DLY}, \cite{DLWY},  \cite{DL}, \cite{DW1} the construction of vertex operator algebras $N(\g,k)$ and their generators.

Let $M_{\widehat{\h}}(k)$ be the vertex operator subalgebra of $V_{\widehat{\g}}(k,0)$
generated by $h(-1)\1$ for $h\in \mathfrak h$ with the Virasoro
element
$$\omega_{\mathfrak h} = \frac{1}{2k}
\sum_{i=1}^{l}u_i(-1)u_i(-1)\1$$
of central charge $l.$ For $\lambda\in
{\mathfrak h}^*,$ denote by  $M_{\widehat{\h}}(k,\lambda)$ the irreducible
highest weight module for $\wh$ with a highest weight vector
$e^\lambda$ such that $h(0)e^\lambda = \lambda(h) e^\lambda$ for
$h\in \mathfrak h.$

Let $N(\g,k)=\{v\in V_{\widehat{\g}}(k,0)|u_nv=0, u\in M_{\widehat{\h}}(k), n\geq 0\}$ be  the commutant
\cite{FZ} of $M_{\wh}(k)$ in $V_{\widehat{\g}}(k,0)$. It is easy to see that
 $N(\g,k)=\{v\in V_{\widehat{\g}}(k,0)|h(n)v=0, h\in \h, n\geq 0\}.$
The $N(\g,k)$ is a vertex
operator algebra with the Virasoro vector $\omega =
\omega_{\mraff} - \omega_{\mathfrak h}$ whose central charge is
$\frac{k\dim \g}{k+h^{\vee}}-l.$

For $\alpha\in \Delta_+,$  let $k_{\alpha}=\frac{\<\theta,\theta\>}{\<\alpha,\alpha\>}k.$ Clearly, $k_{\alpha}$ is a nonnegative integer.
Set
\begin{equation}\label{eq:w3}
\begin{split}
\omega_{\alpha} &=\frac{1}{2k_{\alpha}(k_{\alpha}+2)}( -k_{\alpha} h_\alpha(-2)\1
-h_\alpha(-1)^{2}\1 +2k_{\alpha}x_{\alpha}(-1)x_{-\alpha}(-1)\1),
\end{split}
\end{equation}
\begin{equation}\label{eq:W3}
\begin{split}
W_{\alpha}^3 &= k_{\alpha}^2 h_\alpha(-3)\1 + 3 k_{\alpha} h_\alpha(-2)h_\alpha(-1)\1
+
2h_\alpha(-1)^3\1 \\
&-6k_{\alpha} h_\alpha(-1)x_{\alpha}(-1)x_{-\alpha}(-1)\1
 +3 k_{\alpha}^2x_{\alpha}(-2)x_{-\alpha}(-1)\1 -3
k_{\alpha}^2x_{\alpha}(-1)x_{-\alpha}(-2)\1.
\end{split}
\end{equation}
One can also see \cite{DLY}, \cite{DW1} for the definition of the highest vector $W_\alpha^4,$ $W_\alpha^5$
of weights 4 and 5.
Then $N(\g,k)$ is generated by $\omega_{\alpha}$, $W_{\alpha}^3$ for $\alpha\in\Delta_+$ \cite{DW1} and
 the subalgebra of $N(\g,k)$ generated by $\omega_{\alpha}$ and $W_\alpha^3$ for fixed $\alpha$ is isomorphic to
$N(sl_2,k_{\a})$ (cf. \cite{DLY}, \cite{DLWY}, \cite{DW1}).

\begin{rem} It is proved in \cite{DLY} that $N(sl_2,k)$ is strongly generated by $\omega, W^3, W^4, W^5$
where we omit the $\alpha$ as there is only one positive root. But it is not clear if $N(\g,k)$ is strongly generated by $\omega_{\alpha}$, $W_{\alpha}^i$ for $\alpha\in\Delta_+$ and $i=3,4,5.$ We also do not know the commutator relation $[Y(u,z_1),Y(v,z_2)]$ in general for these generators.
It is definitely important to find the relation $[Y(u,z_1),Y(v, z_2)]$ for the generators explicitly. This will help to understand the structure of $N(\g,k)$ better.
\end{rem}

For  $\Lambda, \lambda \in {\mathfrak h}^*$, set
\begin{equation*}
V_{\widehat{\g}}(k,\Lambda)(\lambda)=\{v\in V_{\widehat{\g}}(k,\Lambda)|h(0)v=\lambda(h) v, \forall\;
h\in\mathfrak h\}.
\end{equation*} Then we have
\begin{equation}\label{eq:V-dec}
V_{\widehat{\g}}(k,\Lambda)=\oplus_{\lambda\in Q+\Lambda}V_{\widehat{\g}}(k,\Lambda)(\lambda).
\end{equation}
Note that $M_{\widehat{\h}}(k)\otimes N(\g,k)$ is a vertex operator subalgebra of $V_{\wg}(k,0)$ isomorphic to  $V_{\widehat{\g}}(k,0)(0)$ and
$V_{\widehat{\g}}(k,\Lambda)(\lambda)=M_{\widehat{\h}}(k,\lambda)\otimes N^{\Lambda,\lambda}$ as a module for $M_{\widehat{\h}}(k)\otimes N(\g,k)$ where
\begin{equation*}
N^{\Lambda, \lambda}= \{ v \in V_{\widehat{\g}}(k,\Lambda)\,|\, h(m)v =\lambda(h)\delta_{m,0}v
\text{ for }  h\in \mathfrak h, m \ge 0\}
\end{equation*}
is the space of highest weight vectors with highest weight $\lambda$ for
$\wh.$ Clearly, $N(\g,k)=N^{0,0}.$

For any $\alpha\in \Delta,$ the subalgebra of $N(\g,k)$ generated by $\omega_{\alpha}, W^3_{\alpha}$
is isomorphic to $N(\g^{\alpha},k_{\alpha})$ and the subalgebra of $V_{\widehat{\g}}(k,0)(0)$ generated by
$\omega_{\alpha}, W^3_{\alpha}, t_{\alpha}(-1)\1$ is isomorphic to $V_{\widehat{\g^\alpha}}(k_\alpha,0)(0)$
\cite{DW1}. Consequently, we regard $N(\g,k_{\alpha})$ as a subalgebra of $N(\g,k),$
and $V_{\widehat{\g^\alpha}}(k_\alpha,0)(0)$ as a subalgebra of $V_{\widehat{\g}}(k,0)(0).$
Here is a stronger result on the generators of $N(\g,k).$ Recall that $\{\alpha_1,...,\alpha_l\}$ are the simple roots.

\begin{prop}\label{generator} The vertex operator algebra $N(\g,k)$ is generated by  $N(\g^{\alpha_i},k_{\alpha_i})$ or $\omega_{\alpha_i}$, $W_{\alpha_i}^3$
for $i=1,...,l.$
\end{prop}
\begin{proof} From \cite{DW1} we know that the vertex operator subalgebra $V_{\widehat{\g}}(k,0)(0)=M_{\widehat{\h}}(k)\otimes N(\g,k)$ is generated by $t_{\alpha}(-1)\1$ and $x_{-\alpha}(-2)x_{\alpha}(-1)\1$ for $\alpha\in\Delta_{+}.$
We first prove that $V_{\widehat{\g}}(k,0)(0)$ is generated by $t_{\alpha_{i}}(-1)\1$ and $x_{-\alpha_i}(-2)x_{\alpha_i}(-1)\1$ for $1\leq i \leq l.$ In other words, we only need the simple roots
for the generators.

Since every positive root can be generated from simple roots we see that $V_{\widehat{\g}}(k_{\alpha},0)(0)$ is spanned
by
$$a_1(-m_1)\cdots a_s(-m_s)x_{\beta_{1}}(-n_{1})x_{\beta_{2}}(-n_{2})\cdots
x_{\beta_{t}}(-n_{t})\1$$
where $a_i\in \mathfrak h, \beta_j\in\{\pm\alpha_1,...,\pm\alpha_l\}, m_i>0, n_j\geq 0$ and
$\beta_{1}+\beta_{2}+\cdots+\beta_{t}=0.$
By Proposition 4.5.8 of \cite{LL} we see that
$V_{\widehat{\g}}(k,0)(0)$  is spanned by
$$a_1(-m_1)\cdots a_s(-m_s)x_{\beta_{1}}(n_{1})x_{-\beta_{1}}(p_{1})
x_{\beta_{2}}(n_{2})x_{-\beta_{2}}(p_{2})\cdots
x_{\beta_{t}}(n_{t})x_{-\beta_{t}}(p_{t})\1$$
where $a_i, m_i$ are as before and $\beta_j\in \{\alpha_1,...,\alpha_l\},$  $n_j, p_j\in\Z.$
It follows from Proposition 4.5.7 \cite{LL} that  $V_{\widehat{\g}}(k,0)(0)$
is spanned by
$$a_1(-m_1)\cdots a_s(-m_s)u^1_{n_1}\cdots u^t_{n_t}\1$$
where $u^j\in V_{\widehat{\g^{\beta_j}}}(k_{\beta_j},0)(0),$ $\beta_i\in \{\alpha_1,...,\alpha_l\},$
and $n_j\in\Z.$

Since each $V_{\widehat{\g^{\alpha_i}}}(k_{\alpha_i},0)(0)$ is generated by $t_{\alpha_i}(-1)$
and $\omega_{\alpha_i}, W^3_{\alpha_i}$ we see immediately that $N(\g,k)$ is
generated by  $\omega_{\alpha_i}$, $W_{\alpha_i}^3$ or $N(\g^{\alpha_i},k_{\alpha_i})$
for $i=1,...,l.$ The proof is complete.
\end{proof}

Next we discuss the decomposition of $N^{\Lambda,\lambda}$ into the direct sum of weight spaces.
Consider the weight space decomposition $L_\g(\Lambda)=\bigoplus_{\lambda\in {\mathfrak h}^*}L_\g(\Lambda)_{\lambda}$ where $L_\g(\Lambda)_{\lambda}$ is the weight space of $L_\g(\Lambda)$ with weight $\lambda.$
Let $P(L_\g(\Lambda))=\{\lambda\in {\mathfrak h}^*|L_\g(\Lambda)_{\lambda}\ne 0\}$ be the weights of $L_\g(\Lambda).$
\begin{lem}\label{l3.1} Let $\lambda\in P(L_\g(\Lambda)).$ Then the $N(\g,k)$-module
\begin{equation}\label{3.1+}
N^{\Lambda,\lambda}=\bigoplus_{n\geq 0}N^{\Lambda,\lambda}_{n_\Lambda-\frac{\<\lambda,\lambda\>}{2k}+n}
\end{equation}
is generated by the irreducible $A(N(\g,k))$-module $N^{\Lambda,\lambda}_{n_\Lambda-\frac{\<\lambda,\lambda)}{2k}}=L_\g(\Lambda)_\lambda$ where
$$N^{\Lambda,\lambda}_{n_\Lambda-\frac{\<\lambda,\lambda\>}{2k}+n}=\{w\in N^{\Lambda,\lambda}|L(0)w=(n_\Lambda-\frac{\<\lambda,\lambda\>}{2k}+n)w\}$$
and $Y(\omega,z)=\sum_{n\in\Z }L(n)z^{-n-2}.$
\end{lem}
\begin{proof} Write $Y(\omega_{\mathfrak h},z)=\oplus_{n\in\Z} L_{\mathfrak h}(n)z^{-n-2}$ and note that $L_{\mathfrak h}(0)$ acts on $L_\g(\Lambda)_\lambda$ as $\frac{\<\lambda,\lambda\>}{2k}.$ The decomposition
(\ref{3.1+}) follows immediately.

Next we prove that $N^{\Lambda,\lambda}$ is generated by $L_\g(\Lambda)_\lambda$ which is an irreducible  $A(N(\g,k))$-module for $\lambda\in P(L_\g(\lambda)).$
We observe that $V_{\widehat{\g}}(k,\Lambda)$ is generated by any nonzero vector in $L_\g(\Lambda)$
from the construction of $V_{\widehat{\g}}(k,\Lambda)$ as $L_\g(\Lambda)$ is an irreducible $\g$-module. It follows from \cite{DM} and \cite{L1} that for any nonzero vector $w\in L_\g(\Lambda)_\lambda,$
$V_{\widehat{\g}}(k,\Lambda)$ is spanned by $u_nw$ for $u\in V_{\widehat{\g}}(k,0)$ and $n\in\Z.$ Since $u_nw\in V_{\widehat{\g}}(k,\Lambda)(\alpha+\lambda)$ for $u\in V_{\widehat{\g}}(k,0)(\alpha)$ for $\alpha\in Q,$ $V_{\widehat{\g}}(k,\Lambda)(\lambda)$ is spanned
by $u_nw$ for $u\in M_{\widehat{\h}}(k)\otimes N(\g,k)$ and $n\in \Z.$ It follows that $N^{\Lambda,\lambda}$ is spanned by $u_nw$ for $u\in N(\g,k)$ and $n\in \Z.$ This implies that  $N^{\Lambda,\lambda}$ is generated by $L_\g(\Lambda)_\lambda$ and $L_\g(\Lambda)_\lambda$ is a simple $A(M_{\widehat{\h}}(k))\otimes N(\g,k))$-module. As $A(M_{\widehat{\h}}(k)\otimes N(\g,k))$ is isomorphic to $A(M_{\widehat{\h}}(k))\otimes_{\C}A(N(\g,k))$ \cite{DMZ} and
$A(M_{\widehat{\h}}(k))$ is commutative, we conclude that $L_\g(\Lambda)_\lambda$ is an irreducible $A(N(\g,k))$-module.
\end{proof}

The universal enveloping algebra $U(\wg)$ is $Q$-graded:
\begin{equation}\label{ugd}
U(\wg)=\bigoplus_{\alpha\in Q}U(\wg)(\alpha)
\end{equation}
where $U(\wg)(\alpha)=\{v\in U(\wg)|[h(0),v]=\alpha(h)v,\ \forall h\in \h\ \}.$ Then $U(\wg)(0)$ is an  associative subalgebra of $U(\wg)$ and  is generated by $K$, $h(m),$ $x_{\alpha_1}(s_1)\cdots x_{\alpha_n}(s_n)$
for $h\in\h, \alpha_i\in\Delta$ with $\sum_i\alpha_i=0$ and $m,s_i\in\Z.$ Note that $U(\wg)=\bigoplus_{n\in\Z}U(\wg)_n$ is also $\Z$-graded such that $\deg x(n)=-n$ for $x\in\g$ and $n\in \Z,$ and $\deg K=0.$ This induces a
$\Z$-gradation on $U(\wg)(0).$ It is clear that for $\Lambda\in P_+^k$ and $\lambda\in P(L_\g(\Lambda)),$ $V_{\widehat{\g}}(k,\Lambda)(\lambda)=M_{\widehat{\h}}(k,\lambda)\otimes N^{\Lambda,\lambda}$ is a $U(\wg)(0)$-module
and $V_{\widehat{\g}}(k,\Lambda)(\lambda)_{n_\Lambda}=L_\g(\Lambda)_{\lambda}$ is an irreducible $U(\wg)(0)_0$-module where $V_{\widehat{\g}}(k,\Lambda)(\lambda)_{n_\Lambda}=V_{\widehat{\g}}(k,\Lambda)(\lambda)\cap V_{\widehat{\g}}(k,\Lambda)_{n_\Lambda}.$
The following Lemma is immediate now.
\begin{lem}\label{l3.2} Let $\Lambda\in P_+^k$ and $\lambda\in P(L_\g(\Lambda)).$ Then $V_{\widehat{\g}}(k,\Lambda)(\lambda)=M_{\widehat{\h}}(k,\lambda)\otimes N^{\Lambda,\lambda}$ is a Verma $U(\wg)(0)$-module generated by $e^{\lambda}\otimes L_\g(\Lambda)_{\lambda}$
in the sense that
$Ie^{\lambda}\otimes L_\g(\Lambda)_{\lambda}=0$ where  $I$ is the left ideal of $U(\wg)(0)$ generated by $h(n), x_{\alpha_1}(s_1)\cdots x_{\alpha_p}(s_p)$ for $h\in \h, \alpha_i\in \Delta$ with $\sum_i\alpha_i=0,$ $n,s_p> 0,$ and
any $U(\wg)(0)$ module $W$ generated by $e^{\lambda}\otimes L_\g(\Lambda)_{\lambda}$ with  $Ie^{\lambda}\otimes L_\g(\Lambda)_{\lambda}=0$ is a quotient of
$V_{\widehat{\g}}(k,\Lambda)(\lambda).$
\end{lem}

We next say few words on certain Verma type modules for vertex operator algebra $N(\g,k).$ Let $V$ be a vertex operator algebra and $A(V)$ be its Zhu algebra. Let $U$ be a simple $A(V)$-module.  Recall from
\cite{DLM2} and \cite{DJ} that the Verma type admissible $V$-module $M(U)=\bigoplus_{n\geq 0}M(U)(n)$ generated by $M(U)(0)=U$ has the property: any admissible $V$-module
$W=\bigoplus_{n\geq 0}W(n)$ generated by $W(0)=U$ is a quotient of $M(U).$
We firmly believe that $N^{\Lambda,\lambda}$ is a Verma type admissible module for $N(\g,k)$ generated by
$L_\g(\Lambda)_\lambda$ for $\Lambda\in P_+^k$ and $\lambda\in P(L_\g(\Lambda)).$ Unfortunately we cannot prove this result in the paper.

\section{Vertex operator algebras $K(\g,k)$}
\setcounter{equation}{0}

We study the parafermion vertex operator algebra $K(\g,k)$ in this section. We get a list of irreducible modules $M^{\Lambda,\lambda}$ from the irreducible $L_{\wg}(k,0)$-modules $L_{\wg}(k,\Lambda)$ for $\Lambda\in P_+^k$ and $\lambda\in \Lambda+Q.$ We also discuss how to use the  lattice vertex operator subalgebra of $L_{\wg}(k,0)$ and the simple currents of $L_{\wg}(k,0)$ to give identifications between different $M^{\Lambda,\lambda}.$

Recall the irreducible quotient $L_{\widehat{\g}}(k,\Lambda)$ of $V_{\widehat{\g}}(k,\Lambda)$ for $\Lambda\in P^k_+.$  Again, the Heisenberg vertex operator algebra $M_{\widehat{\h}}(k)$ generated by
$h(-1)\1$ for $h\in \mathfrak h$ is a simple subalgebra of
$L_{\widehat{\g}}(k,0)$ and $L_{\widehat{\g}}(k,\Lambda)$ is a completely reducible
$M_{\widehat{\h}}(k)$-module. We have a decomposition
\begin{equation}\label{4.1}
L_{\widehat{\g}}(k,\Lambda) = \oplus_{\lambda\in Q+\Lambda} M_{\widehat{\h}}(k,\lambda) \otimes M^{\Lambda,\lambda}
\end{equation}
as modules for $M_{\widehat{\h}}(k)$, where
\begin{equation*}
M^{\Lambda,\lambda} = \{v \in L_{\widehat{\g}}(k,\Lambda)\,|\, h(m)v =\lambda(h)\delta_{m,0}v
\text{ for }\; h\in {\mathfrak h},
 m \ge 0\}.
\end{equation*}
Moreover, $M^{\Lambda,\lambda}=N^{\Lambda,\lambda}/(N^{\Lambda,\lambda}\cap {\cal J}(k,\Lambda)).$
Recall (\ref{3.1+}). We also can write decomposition (\ref{4.1}) as
\begin{equation*}
L_{\widehat{\g}}(k,\Lambda) = \oplus_{\lambda\in Q+\Lambda} L_{\widehat{\g}}(k,\Lambda)(\lambda)
\end{equation*}
where $L_{\widehat{\g}}(k,\Lambda)(\lambda)=M_{\widehat{\h}}(k,\lambda) \otimes M^{\Lambda,\lambda}$
is the weight space of $L_{\widehat{\g}}(k,\Lambda) $ for Lie algebra $\g$ with weight $\lambda.$
Moreover, each $L_{\widehat{\g}}(k,\Lambda)(\lambda)$ is a module for $U(\wg)(0).$

Set $K(\g,k)=M^{0,0}.$ Then $K(\g,k)$ is the commutant of
$M_{\widehat{\h}}(k)$ in $L_{\widehat{\g}}(k,0)$ and is called the parafermion vertex
operator algebra associated to the irreducible highest weight
module $L_{\widehat{\g}}(k,0)$ for $\widehat{\g}.$ The Virasoro vector of
$K(\g,k)$ is given by
$$\omega =\omega_{\mraff} - \omega_{\h},$$
where we still use $\omega_{\mraff}, \omega_{\mathfrak h}$ to
denote their images in $L_{\wg}(k,0)$. Since the Virasoro
vector of $M_{\widehat{\h}}(k)$ is $\omega_{\h}$, we have
$$
 K(\g,k)=\{v \in L_{\wg}(k,0)\,|\, (\omega_{\h})_{0}v =0\}.$$

Note that $K(\g,k)$ is a quotient of $N(\g,k).$ We still denote by $\omega_{\alpha}, W^3_{\alpha}$ for their
images in $K(\g,k).$ It follows from \cite{DW1} that the subalgebra of $K(\g,k)$ generated by $\omega_{\alpha}, W^3_{\alpha}$ is isomorphic to $K(\g^{\alpha},k_{\alpha}).$ So we can regard $K(\g^{\alpha},k_{\alpha})$ as subalgebra of $K(\g,k).$ The following result is an immediate consequence of Proposition \ref{generator}
\begin{prop}\label{generator1}
The vertex operator algebra $K(\g,k)$ is generated by $K(\g^{\alpha_i},k_{\alpha_i})$ or  $\omega_{\alpha_i}$, $W_{\alpha_i}^3$
for $i=1,...,l.$
\end{prop}

 We now turn our attention to the irreducible $K(\g,k)$-modules.
 \begin{lem}\label{l4.1} Let $\Lambda\in P^k_+$ and $\lambda\in \Lambda+Q.$ Then $M^{\Lambda,\lambda}$ is an irreducible $K(\g,k)$-module.
\end{lem}
\begin{proof} Since $M_{\wh}(k)\otimes K(\g,k)=L_{\wg}(k,0)(0)$ and $M_{\wh}(k,\lambda)$ is an irreducible
$M_{\wh}(k)$-module, it is good enough to show that $M_{\wh}(k,\lambda)\otimes M^{\Lambda,\lambda}=L_{\wg}(k,\Lambda)(\lambda)$ is
an irreducible $L_{\wg}(k,0)(0)$-module. Note that $L_{\wg}(k,\Lambda)$ is an irreducible $L_{\wg}(k,0)$-module. It follows from \cite{DM}, \cite{LL}, \cite{L1}  that for any nonzero $w\in L_{\wg}(k,\Lambda)(\lambda),$
$L_{\wg}(k,\Lambda)$ is spanned by $u_nw$ for $u\in L_{\wg}(k,0)(\alpha)$ with $\alpha\in Q,$ and $n\in\Z.$
Clearly,  $u_nw\in L_{\wg}(k,\Lambda)(\lambda+\alpha).$ This implies that $L_{\wg}(k,\Lambda)(\lambda)$ is spanned
by $u_nw$ for $u\in L_{\wg}(k,0)(0)$ and $n\in\Z.$ That is, $L_{\wg}(k,\Lambda)(\lambda)$ is
an irreducible $L_{\wg}(k,0)(0)$-module.
\end{proof}

But not all these irreducible modules $M^{\Lambda,\lambda}$ are different. We give some identifications of these modules using lattice vertex operator algebras of $L_{\wg}(k,0)$ and simple currents of $L_{\wg}(k,0).$

Recall the root lattice $Q$ of $\g.$ Let $Q_L$ be the sublattice of $Q$ spanned by the long roots.
Then $Q$ and $Q_L$ have the same rank. It is known from \cite{DW3}, \cite{K} that the lattice
vertex operator algebra $V_{\sqrt{k}Q_L}$ (see \cite{B} and \cite{FLM}) is a subalgebra of $L_{\widehat{\g}}(k,0).$
Moreover, $M_{\widehat{\h}}(k)$ is a subalgebra of $V_{\sqrt{k}Q_L}$ with the same Virasoro element. In fact,
$$V_{\sqrt{k}Q_L}=M_{\widehat{\h}}(k)\otimes \C[\sqrt{k}Q_L]=\oplus_{\alpha\in Q_L}M_{\widehat{\h}}(k,\sqrt{k}\alpha)$$
as a module for $M_{\widehat{\h}}(k)$ (see (\ref{4.1})). It is important to point out that in the definition of vertex operator algebra $M_{\widehat{\h}}(k)$ we use
the bilinear form $\<,\>.$ If we use the standard notation for the lattice vertex operator algebra $V_{\sqrt{k}Q_L}=M(1)\otimes \C[\sqrt{k}Q_L]$
we have to use another bilinear form  $(,)=k\<,\>.$ The actions of $h(0)$ on $V_{\sqrt{k}Q_L}$-modules
for  $h\in{\mathfrak h}$ also uses   $(,)$ instead of $\<,\>.$

\begin{rem} The parafermion vertex operator algebra $K(\g,k)$ is also the commutant of the rational vertex operator
algebra $V_{\sqrt{k}Q_L}$ \cite{D}, \cite{DLM1} in rational vertex operator algebra $L_{\wg}(k,0).$ This explains why one expects
that $K(\g,k)$ is rational.
\end{rem}

Let $L$ be an even lattice. As usual, we denote the dual lattice of $L$ by $L^{\circ}.$ It is known from \cite{D} that the irreducible $V_L$-modules are given by $V_{L+\lambda}$ where $\lambda\in L^{\circ}.$ Moreover,  $V_{L+\lambda}= V_{L+\mu}$ if $\lambda-\mu\in L.$ Consider the coset decomposition $Q=\cup_{i\in Q/kQ_L}(kQ_L+\beta_i).$ Since $V_L$ is rational for any positive definite even lattice  $L$ (see \cite{D}, \cite{DLM1}),
we have the decomposition
\begin{equation}\label{4.2}
L_{\widehat{\g}}(k,\Lambda)=\bigoplus_{i\in Q/kQ_L}V_{\sqrt{k}Q_L+\frac{1}{\sqrt{k}}(\Lambda+\beta_i)}\otimes M^{\Lambda,\Lambda+\beta_i}
\end{equation}
as modules for $V_{\sqrt{k}Q_L}\otimes K(\g,k)$ where $M^{\Lambda,\lambda}$ is as before. Again, as a $M_{\widehat{\h}}(k)$-module
\begin{equation}\label{4.3}
V_{\sqrt{k}Q_L+\frac{1}{\sqrt{k}}(\Lambda+\alpha)}=\oplus_{\beta\in Q_L}M_{\widehat{\h}}(k,k\beta+\Lambda+\alpha).
\end{equation}
That is, for $h\in {\mathfrak h},$ $h(0)$ acts on  $V_{\sqrt{k}Q_L}\otimes K(\g,k)$-module $V_{\sqrt{k}Q_L+\frac{1}{\sqrt{k}}\lambda}\otimes M^{\Lambda,\lambda} $ is given by
$$h(0)(u\otimes e^{\sqrt{k}\beta+\frac{1}{\sqrt{k}}\lambda})\otimes w=\<h,k\beta+\lambda\>(u\otimes e^{\sqrt{k}\beta+\frac{1}{\sqrt{k}}\lambda})\otimes w$$
for $\beta\in Q_L,$ $\lambda\in \Lambda+Q,$ $u\in M_{\widehat{\h}}(k)$ and $w\in M^{\Lambda,\lambda}.$

Here is our first identification among $M^{\Lambda,\lambda}.$
In the case $\g=sl_2,$ this result has been obtained in \cite{DLY}.
\begin{prop}\label{l4.1'} Let $\Lambda\in P^k_+$ and $\lambda\in \Lambda+Q.$ Then
$M^{\Lambda, \lambda+k\beta}$ and $M^{\Lambda, \lambda}$ are isomorphic for any $\beta\in Q_L.$
\end{prop}
\begin{proof} (1) follows from the decompositions (\ref{4.1})-(\ref{4.3}).
\end{proof}

We next investigate  more connection between different $M^{\Lambda,\lambda}$ and $M^{\Lambda',\lambda'}.$ For this purpose, we need to discuss the simple currents for the vertex operator algebra $L_{\widehat{\g}}(k,0)$ following \cite{L4}.

Let $\Lambda_1,...,\Lambda_l$ be the fundamental weights of $\g.$ Then $P=\oplus_{i=1}^l\Z\Lambda_i$ is the weight lattice. Let $\theta=\sum_{i=1}^la_i\alpha_i.$  Here is a list of $a_i=1$ using the labeling from \cite{K}:
\begin{eqnarray*}
A_l: & a_1,...,a_l\\
B_l: & a_1\\
C_l: &a_l\\
D_l: &a_1, a_{l-1}, a_l\\
E_6: &a_1, a_5\\
E_7: &a_6
\end{eqnarray*}
There are  $|P/Q|-1$ such $i$ with  $a_i=1$ \cite{L4}.

It is proved in \cite{L2'} and \cite{L4} that $L_{\wg}(k,k\Lambda_i)$ are simple current if $a_i=1.$ To see this
we let $h^i\in \h$ for $i=1,...,l$ defined by $\alpha_i(h^j)=\delta_{i,j}$ for $j=1,...,l.$ For any $h\in \h$
set
$$\Delta(h,z)=z^{h(0)}\exp\left(\sum_{n=1}^{\infty}\frac{h(n)(-z)^{-n}}{-n}\right).$$
 The following result was obtained in \cite{L2'} and \cite{L4}.
\begin{thm} Assume $a_i=1.$

(1) For any $\Lambda\in P_+^k$, $L_{\wg}(k,\Lambda)^{(h^i)}=(L_{\wg}(k,\Lambda)^{(h^i)},Y_i)$ is an irreducible
$L_{\wg}(k,0)$-module where $L_{\wg}(k,\Lambda)^{(h^i)}=L_{\wg}(k,\Lambda)$ as vector spaces
and $Y_i(u,z)=Y(\Delta(h^i,z)u,z)$ for $u\in L_{\wg}(k,0).$
 Let $\Lambda^{(i)}\in P_+^k$ such that $L_{\wg}(k,\Lambda)^{(h^i)}$
is isomorphic to $L_{\wg}(k,\Lambda^{(i)}).$

(2) The $L_{\wg}(k,0^{(i)})=L_{\wg}(k,k\Lambda_i)$ is a simple current and
$L_{\wg}(k,k\Lambda_i)\boxtimes L_{\wg}(k,\Lambda)=L_{\wg}(k,\Lambda^{(i)})$ for all $\Lambda\in P_+^k.$
\end{thm}

Although we do not need to know $\Lambda^{(i)}$ explicitly in this paper, it is still an interesting problem
to find out. In the case $\g=sl_2,$ assume $\Delta=\{\pm \alpha\}.$ Then
$P_+^k=\{\frac{s\alpha}{2}|s=0,...,k\}$ and $(\frac{s\alpha}{2})^{(1)}=
\frac{(k-s)\alpha}{2}$ where $1$ corresponds to $h^1.$

Recall decomposition (\ref{4.1}).  We now investigate how $\h$ acts on each weight space $L_{\wg}(k,\Lambda)(\lambda)$
regarding as a subspace of $L_{\wg}(k,\Lambda)^{(h^i)}.$ This result will be helpful
in the identification of irreducible $K(\g,k)$-modules later.

Note that for $h\in\h,$  $\Delta(h^i,z)h(-1)\1=h(-1)\1+k\<h^i,h\>z^{-1}.$ Thus
$$Y_i(h(-1)\1,z)=Y(h(-1)\1,z)+k\<h^i,h\>z^{-1}.$$
In particular, the $h(0)$ acts  $(L_{\wg}(k,\Lambda), Y_i)$
as $h(0)+\<h^i,h\>k$ and acts on $L_{\wg}(k,\Lambda)(\lambda)\subset L_{\wg}(k,\Lambda)^{(h^i)},$
as $\lambda(h)+\<h^i,h\>k$ for $\lambda\in \Lambda+Q.$ Recall the identification between $\h$ and $\h^*.$
We see that $h^i=\frac{2t_{\Lambda_i}}{\<\alpha_i,\alpha_i\>}.$ Note from \cite{K} that if $a_i=1$ then $\alpha_i$ is a long root. So $h^i=t_{\Lambda_i}$ with $a_i=1.$
This implies that
$\<h^i,h\>=\Lambda_i(h)$ and $$\lambda(h)+\<h^i,h\>k=\lambda(h)+k\Lambda_i(h).$$
Thus we have proved the following result:
\begin{lem}\label{iden1} There is an $L_{\wg}(k,0)$-module isomorphism
$$f_{\Lambda,i}: L_{\wg}(k,\Lambda)^{(h^i)} \to L_{\wg}(k,\Lambda^{(i)})$$
such that
$f_{\Lambda,i}(L_{\wg}(k,\Lambda)(\lambda))=L_{\wg}(k,\Lambda^{(i)})(\lambda+k\Lambda_i)$
for all $\lambda\in \Lambda+Q.$
\end{lem}

It is easy to see that $\Delta(h^i,z)u=u$ for
$u\in K(\g,k).$ Thus $Y_i(u,z)=Y(u,z)$ for $u\in K(\g,k)$ and $f_{\Lambda,i}:L_{\wg}(k,\Lambda)\to L_{\wg}(k,\Lambda^{(i)})$ is a $K(\g,k)$-module isomorphism. This gives us the second identification.
\begin{thm}\label{iden2} We have a $K(\g,k)$-module isomorphism between $M^{\Lambda, \lambda}$ and
$M^{\Lambda^{(i)},\lambda+k\Lambda_i}$ for any $\lambda\in \Lambda+Q.$
Moreover, $\Lambda_i$ does not lie in $Q_L.$ That is, this identification
is different from the identification
given in Propsition \ref{l4.1'}.
\end{thm}
\begin{proof} The identification between  $M^{\Lambda, \lambda}$ and
$M^{\Lambda^{(i)},\lambda+k\Lambda_i}$
for  $\lambda\in \Lambda+Q$ is an immediate consequence of Lemma \ref{iden1}.

To prove the identification here is different from that given in Proposition \ref{l4.1'}, it is sufficient
to show that $\Lambda_i$ does not lie in $Q_L.$ This is clear if $\g$ is of $A,D,E$ type
as $P=Q\cup \cup_{i,a_i=1}(Q+\Lambda_i).$ So we only need to deal with type $B_l$ and  $C_l.$

For type $B_l,$  let $\E=\R^l$ with the standard orthonormal basis $\{\epsilon_1,...,\epsilon_l\}.$ Then
 $$\Delta=\{\pm\epsilon_i, \pm(\epsilon_i\pm\epsilon_j)|i\ne j\}$$
 and $\alpha_1=\epsilon_1-\epsilon_2,...,\alpha_{l-1}=\epsilon_{l-1}-\epsilon_l, \alpha_l=\epsilon_l.$
 In this case only $a_1=1.$ Obviously, $\Lambda_1=\epsilon_1$ is not an element of $Q_L.$

 For type $C_l,$ let $\E$ be the same as before. Then $$\Delta=\{\pm\sqrt{2}\epsilon_i, \pm\frac{1}{\sqrt{2}}(\epsilon_i\pm\epsilon_j)|i\ne j\}$$
 and $\alpha_1=\frac{\epsilon_1-\epsilon_2}{\sqrt{2}},\cdots, \alpha_{l-1}=\frac{\epsilon_{l-1}-\epsilon_l}{\sqrt{2}}, \alpha_l=\sqrt{2}\epsilon_l.$
 We know only $a_l=1.$  Then
 $$\Lambda_l=\frac{\epsilon_1+\cdots+\epsilon_l}{\sqrt{2}}$$
 and $Q_L=\sum_{i=1}^l\Z\sqrt{2}\epsilon_i.$  It is easy to see that $\Lambda_l$ does not belong to $Q_L.$ \end{proof}

\section{Representations of $K(\g,k)$}
\setcounter{equation}{0}

In this section we  establish  the rationality for vertex operator algebra for $K(\g,k)$ and  determine the irreducible modules for  $K(\g,k).$

\begin{thm}\label{rational} Let $\g$ be any finite dimensional simple Lie algebra and $k$ be any positive integer. Then

(1) The parafermion vertex operator algebra $K(\g,k)$ is rational.

(2) Any irreducible $K(\g,k)$-module is isomorphic to $M^{\Lambda,\lambda}$ for some $\Lambda\in P_+^k$ and $\lambda\in \Lambda+Q.$ Moreover $M^{\Lambda,\lambda}$ and $M^{\Lambda, \lambda+k\beta}$ are isomorphic for any $\beta\in Q_L,$ and  $M^{\Lambda,\lambda}$ and $M^{\Lambda^{(i)},\lambda+k\Lambda_i}$ are isomorphic for any $i$ with $a_i=1.$
\end{thm}

\begin{proof} (1) Let $G$ be the dual group of the abelian group $Q/kQ_L.$ Then $G$ is a finite subgroup
of automorphisms of  $L_{\wg}(k,0)$ such that $g\in G$ acts as $g(\beta_i+kQ_L)$ on $V_{\sqrt{k}Q_L+\frac{1}{\sqrt{k}}\beta_i}\otimes M^{0,\beta_i}.$ In other words, each $\beta_i+kQ_L$
is an irreducible character of $G.$  So $V_{\sqrt{k}Q_L+\frac{1}{\sqrt{k}}\beta_i}\otimes M^{0, \beta_i}$ in the decomposition
$$L_{\wg}(k,0)=\bigoplus_{i\in Q/kQ_L}V_{\sqrt{k}Q_L+\frac{1}{\sqrt{k}}\beta_i}\otimes M^{0, \beta_i}$$
corresponds to the character $\beta_i+kQ_L.$
In particular, $L_{\wg}(k,0)^G=V_{\sqrt{k}Q_L}\otimes K(\g,k).$

Since $G$ is a finite abelian group, it follows from Theorem 1 of \cite{M}, Theorem 5.24 of \cite{CM}, $V_{\sqrt{k}Q_L}\otimes K(\g,k)$ is rational and $C_2$-cofinite. As $V_{\sqrt{k}Q_L}$ is rational, the rationality of $K(\g,k)$ follows immediately.

(2) Since  $L_{\wg}(k,0),$ $V_{\sqrt{k}Q_L},$ $K(\g,k)$ are rational, $C_2$-cofinite, CFT type,  and  $K(\g,k)$, $V_{\sqrt{k}Q_L}$ are commutants each other in  $L_{\wg}(k,0),$ we know from Theorem \cite{KM} that
every irreducible $K(\g,k)$-module occurs in an irreducible $L_{\wg}(k,0)$-module. The rest follows from Proposition \ref{l4.1'} and Theorem \ref{iden2}.
\end{proof}

A complete identification of irreducible $K(\g,k)$-modules is given in \cite{ADJR} by using the quantum dimensions and global dimensions.

\end{document}